\documentclass[12pt]{article}

\usepackage[psamsfonts]{amsfonts}
\usepackage{amsmath,amssymb,amsthm}
\usepackage{hyperref}

\newcommand{\Degree}{K}

\numberwithin{equation}{section}

\theoremstyle{plain}
\newtheorem{theorem}{Theorem}
\newtheorem{prop}[theorem]{Proposition}

\newtheorem{lemma}[theorem]{Lemma}

\newcommand{\Z}{\ensuremath{\mathbb{Z}}}
\newcommand{\Zd}{\ensuremath{\Z^d}}
\newcommand{\ee}{\ensuremath{\mathrm{e}}}

\begin{document}

\title{
The growth constants of lattice trees
\\
and lattice animals in high dimensions
}

\author{
Yuri Mej\'ia Miranda\thanks{
Department of Mathematics, University of British Columbia,
Vancouver, BC, Canada V6T 1Z2.  Email: {\tt yuri@math.ubc.ca}}
\, and
Gordon Slade\thanks{
Department of Mathematics, University of British Columbia,
Vancouver, BC, Canada V6T 1Z2.  Email: {\tt slade@math.ubc.ca}}
}

\date{December 10, 2010}

\maketitle

\begin{abstract}
We prove that the growth constants for nearest-neighbour lattice trees
and lattice (bond) animals on the integer lattice $\Zd$ are
asymptotic to $2d\ee$ as the dimension goes to infinity,
and that their critical one-point functions converge to $\ee$.
Similar results are obtained in dimensions $d>8$
in the limit of increasingly
spread-out models; in this case the result for the growth
constant is a special case of previous
results of M.~Penrose.
The proof is elementary, once we apply previous results of T.~Hara and G.~Slade
obtained using the lace expansion.
\end{abstract}

\section{The main result}

We define two different regular graphs with vertex set $\Zd$, as follows.
The \emph{nearest-neighbour} graph has edge set consisting
of pairs $\{x,y\}$ with $\|x-y\|_1=1$.
The \emph{spread-out} graph has edge set consisting
of pairs $\{x,y\}$ with $0<\|x-y\|_\infty\le L$, with $L \ge 1$ fixed.
These graphs have degrees $2d$ and $(2L+1)^d-1$,
respectively.  Often we
discuss both graphs simultaneously, and use $\Degree$ to denote the
degree in either case.
Also, we will write $\lim_{\Degree \to \infty}$ to simultaneously denote
the limit as $d \to \infty$ for the nearest-neighbour case,
and the limit as $L \to \infty$ for the spread-out case.

On either graph, a \emph{lattice animal} is a finite connected subgraph,
and a \emph{lattice tree} is a finite connected subgraph
without cycles.  These very natural combinatorial objects are also
fundamental in polymer science \cite{Jans00}.
We denote the number of lattice animals containing $n$ bonds
and containing the origin of $\Zd$ by $a_n$, and the number
of lattice trees containing $n$ bonds
and containing the origin of $\Zd$ by $t_n$.
Standard subadditivity arguments \cite{Klar67,Klei81} provide the existence
of the \emph{growth constants}
\begin{equation}
    \tau = \lim_{n \to \infty} t_n^{1/n},
    \quad\quad
    \alpha = \lim_{n \to \infty} a_n^{1/n}.
\end{equation}
The growth constants
of course depend on $d$, and for the spread-out model, also on $L$.
The \emph{one-point functions}
\begin{equation}
    g(z)=\sum_{n=0}^\infty t_nz^n \quad \text{and} \quad
    g^{(a)}(z)=\sum_{n=0}^\infty a_nz^n
\end{equation}
have radii of convergence $z_c=\tau^{-1}$ and
$z_c^{(a)}=\alpha^{-1}$, respectively.

We will
rely on a result obtained by Hara and Slade \cite{HS90b}
using the lace expansion, but we will not
need any details about the lace expansion in this paper.
It is shown in \cite{HS90b}
that $g(z_c)$ and $g^{(a)}(z_c^{(a)})$
are finite
(in fact, at most  $4$)
for the nearest-neighbour
model in sufficiently high dimensions, and for the spread-out model
in dimensions $d>8$ if $L$ is sufficiently large, and that,
in these two limits, $z_c$ and $z_c^{(a)}$ obey
the equations
\begin{equation}
\label{g1}
    \lim_{\Degree \to \infty} \Degree z_c g(z_c)
    =
    \lim_{\Degree \to \infty} \Degree z_c^{(a)} g^{(a)}(z_c^{(a)})
     = 1.
\end{equation}
This is discussed for the
nearest-neighbour model in \cite{Hara08} (see, in particular,
\cite[(1.31)]{Hara08}), and the same considerations apply for
the spread-out model.  In fact, much more is known \cite{Slad06}.

Our main result is the following theorem.
The asymptotic relation in its statement means that
the limit of the ratio of left- and right-hand sides is equal to $1$.

\begin{theorem}
\label{thm:1}
For the nearest neighbour model as $d \to \infty$,
and for the spread-out model in dimensions $d>8$ as $L \to \infty$,
\begin{equation}
\label{taualphalim}
    \tau \sim \Degree \ee \quad \text{and} \quad \alpha \sim \Degree \ee,
\end{equation}
and, in these same limits,
\begin{equation}
\label{elim}
    \lim_{\Degree \to \infty} g(z_c)
    = \lim_{\Degree \to \infty} g^{(a)}(z_c^{(a)})
    =\ee.
\end{equation}
\end{theorem}

To our knowledge, Theorem~\ref{thm:1} is new for
the nearest-neighbour model.  The proof of
Theorem~\ref{thm:1} is the same for both the
nearest-neighbour and spread-out models.
No bound on the rate of convergence
is obtained here for either \eqref{taualphalim} or \eqref{elim}.
Given \eqref{g1}, the statements $\tau \sim \Degree
\ee$ and $g(z_c) \to \ee$ are equivalent, as are the
statements $\alpha \sim \Degree
\ee$ and $g^{(a)}(z_c^{(a)}) \to \ee$.

Stronger results than \eqref{taualphalim}
have been obtained by Penrose \cite{Penr94} for
the spread-out model using a completely different
method of proof, without restriction to $d>8$ and with the error
estimate
\begin{equation}
\label{Perror}
    \frac{\Degree^\Degree}{(\Degree -1)^{\Degree -1}}
    - O(\Degree^{5/7}\log \Degree)
    \le \tau \le \alpha \le
    \frac{\Degree^\Degree}{(\Degree -1)^{\Degree -1}}
\end{equation}
in all dimensions $d \ge 1$.
Both the right- and left-hand sides of \eqref{Perror} are of course
asymptotic to $\Degree \ee$ as $\Degree \to
\infty$.
When combined with
\eqref{g1}, \eqref{elim} then follows from \eqref{Perror}
for the spread-out model in dimensions $d>8$.

Much stronger results  than \eqref{taualphalim}
have been obtained for the closely
related models of self-avoiding walks and percolation.
Let $c_n$ denote the number of $n$-step self-avoiding walks starting
at the origin.
For nearest-neighbour self-avoiding walks, it was proved in
\cite{HS95} that the connective constant $\mu=\lim_{n\to\infty} c_n^{1/n}$
has an asymptotic expansion $\mu \sim \sum_{i=-1}^\infty m_i (2d)^{-i}$
(as $d \to \infty$),
with $m_i \in \Z$ for all $i$.  The first thirteen coefficients in this
expansion are now known \cite{CLS07}, and it appears likely that
the series $\sum_i m_i x^i$ has radius of convergence equal
to zero.  It may however be Borel summable, and a partial result in
this direction is given in \cite{Grah10}.  Some related results for
nearest-neighbour bond percolation are obtained in \cite{HS95,HS06},
and for spread-out models of percolation and self-avoiding
walks in \cite{HS05y,Penr93,Penr94}.

The behaviour of $\tau$ and $\alpha$ for the nearest-neighbour model,
as $d \to \infty$, has been extensively studied in the physics
literature.  For $\tau$, the expansion
\begin{equation}
\label{taud}
    \tau = \sigma \ee
    \exp
    \left(
    - \frac 12 \frac{1}{\sigma}
    - \frac 83 \frac{1}{\sigma^2}
    - \frac {85}{12} \frac{1}{\sigma^3}
    - \frac {931}{20} \frac{1}{\sigma^4}
    - \frac {2777}{10} \frac{1}{\sigma^5}
    + \cdots
    \right)
    \quad\text{where}\quad \sigma = 2d-1
\end{equation}
is reported in \cite{GP00}, but without a rigorous estimate for the
error term.  This raises the question of whether
there exists an asymptotic expansion for
$\tau$ of the form $\ee \sum_{i=-1}^\infty r_i(2d)^{-i}$,
with $r_i \in {\mathbb Q}$.
For $\alpha$, the series
\begin{align}
\label{alphad}
    \alpha = \sigma \ee
    \exp
    &
    \left(
    - \frac 12 \frac{1}{\sigma}
    - \big(\frac{8}{3}-\frac{1}{2\ee}\big) \frac{1}{\sigma^2}
    - \big(\frac{85}{12}-\frac{1}{4\ee}\big) \frac{1}{\sigma^3}
    - \big(\frac{931}{20}-\frac{139}{48\ee} - \frac{1}{8\ee^2}\big)
    \frac{1}{\sigma^4}
    \right.
    \nonumber \\ &
    \hspace{5mm}
    \left.
    - \big(\frac{2777}{10}+\frac{177}{32\ee} - \frac{29}{12\ee^2}\big)
    \frac{1}{\sigma^5}
    + \cdots
    \right)
\end{align}
was derived in \cite{Harr82,PG95}, again without a rigorous error estimate;
here the role of the transcendental number $\ee$ is more delicate.
Theorem~\ref{thm:1} provides a rigorous confirmation of
the leading terms in \eqref{taud}--\eqref{alphad}.

\section{Main steps in the proof}

We define
\begin{equation}
\label{z0def}
    z_0 = \frac{1}{\Degree \ee}.
\end{equation}
Since $\tau \le \alpha$ by definition, the critical points obey
\begin{equation}
    z_c \ge z_c^{(a)}.
\end{equation}
In fact, the strict inequality $\tau < \alpha$ is known
\cite{Jans00}. The proof of Theorem~\ref{thm:1} uses the following
two ingredients. The content of the first is that $z_c^{(a)} \ge
z_0$, or, equivalently, that $\alpha \le \Degree \ee$.  This fact
is presumably well-known, though we did not find an explicit proof
in the literature.  Klarner \cite{Klar67} proves that for
2-dimensional nearest-neighbour site animals the growth constant
is at most $27/4 = 3^3/2^2$ and Penrose \cite{Penr94} states that
this can be generalised to the upper bound $\alpha \le
\Degree^\Degree/(\Degree - 1)^{\Degree -1} \sim \Degree \ee$ for
bond animals on an arbitrary regular graph. We will provide an
elementary proof that $\alpha \le \Degree \ee$ in
Lemma~\ref{lem:z} below, both to keep self-contained and because
elements of the proof are also useful elsewhere in our approach.

\begin{lemma}
\label{lem:z}
In all dimensions $ d \ge 1$, and for the nearest-neighbour or spread-out models,
\begin{equation}
    z_c \ge z_c^{(a)} \ge z_0 = \frac{1}{\Degree \ee}.
\end{equation}
\end{lemma}

\begin{prop}
\label{prop:z0}
For the nearest-neighbour model, or for the spread-out model in
dimensions $d \ge 1$,
\begin{equation}
    \lim_{\Degree \to \infty}  g(z_0) = \ee.
\end{equation}
\end{prop}

\begin{proof}[Proof of Theorem~\ref{thm:1}]
We will prove that,
under the hypotheses of
Theorem~\ref{thm:1},
\begin{equation}
\label{ge}
    \lim_{\Degree \to \infty}  g(z_c) = \ee.
\end{equation}
It then follows from \eqref{g1}
that $z_c \sim z_0$.  Lemma~\ref{lem:z} then
implies that $z_c^{(a)}\sim z_0$, and finally
\eqref{g1} implies that $\lim_{\Degree \to \infty} g(z_c^{(a)})=\ee$.
Thus Theorem~\ref{thm:1} will
follow, once we prove \eqref{ge}.
By Proposition~\ref{prop:z0} and \eqref{g1},
\begin{equation}
    \lim_{\Degree \to \infty} (\Degree z_cg(z_c)-\ee^{-1} g(z_0)) =0.
\end{equation}
This can be rewritten as
\begin{equation}
    \lim_{\Degree \to \infty} [\Degree (z_c-z_0)g(z_c)+
    \ee^{-1}(g(z_c)-g(z_0))] =0.
\end{equation}
By Lemma~\ref{lem:z} and the monotonicity of $g$, both terms in the limit
are non-negative, and therefore
\begin{equation}
    \lim_{\Degree \to \infty} (g(z_c)-g(z_0)) = 0.
\end{equation}
With Proposition~\ref{prop:z0}, this gives
\eqref{ge} and completes the proof.
\end{proof}

It remains to prove Lemma~\ref{lem:z}
and Proposition~\ref{prop:z0}.

\section{The proof completed}

We will make use of
the following mean-field model
(see \cite{BCHS99,Slad06}), which is related
to the Galton--Watson
branching process with critical Poisson offspring distribution.
In other developments,
the connection with the mean-field model is reflected
by the super-Brownian scaling limits proved
for lattice trees in high dimensions \cite{DS98,Holm08}.

Let ${\mathcal T}_n$ denote the set of $n$-edge rooted plane trees
\cite{Stan97}, and let ${\mathcal T}= \cup_{n=0}^\infty {\mathcal T}_n$.
Given $T \in {\mathcal T}$, we consider mappings
$\varphi : T \to \Zd$ with the property that $\varphi$ maps the root to
the origin, and maps each other vertex of $T$ to a neighbour of its parent
(nearest-neighbour or spread-out, depending on the setting);
the set of such mappings is denoted $\Phi(T)$.
There is no self-avoidance constraint.
By definition, for $T \in {\mathcal T}_n$,
the cardinality of $\Phi(T)$ is $K^n$.
We may interpret the image of $T$ under
$\varphi$ as a multigraph without self-lines,
and we refer to the pair $(T,\varphi)$ as a \emph{mean-field} configuration.
The set of all mean-field configurations $(T,\varphi)$ with
$T \in {\mathcal T}_n$ is denoted ${\mathcal M}_n$.

Let $\xi_i$ denote the forward degree of a vertex $i \in T$; this is
the degree of the root when $i$ is the root of $T$
and otherwise it is the
degree of $i$ minus $1$.  For $n \ge 0$, let
\begin{equation}
\label{fdef}
    f_n = \sum_{(T,\varphi) \in {\mathcal M}_n }
    \prod_{i\in T}\frac{1}{\xi_i!}
    =
    K^n
    \sum_{T  \in {\mathcal T}_n }
    \prod_{i\in T}\frac{1}{\xi_i!},
\end{equation}
where the second equality follows from the fact that
$\Phi(T)$ has $K^n$ elements.
Let $|T|$ denote the number of edges in $T$. Then
\begin{equation}
\label{fsum}
    \sum_{n=0}^\infty f_n  z^{n}
    =
    (K z)^{-1}\sum_{T \in {\mathcal T}}
    (\Degree\ee z)^{|T|+1} \prod_{i\in T}\frac{1}{\ee\xi_i!}.
\end{equation}
Moreover, since
${\mathbb P}(T)=\prod_{i\in T}(\ee\xi_i!)^{-1}$ is the probability
that $T$ arises as the family tree of a  Galton--Watson
branching process with critical Poisson offspring distribution,
it follows from \eqref{fsum} that
\begin{equation}
\label{f1}
    \sum_{n=0}^\infty f_n z_0^n
    =
    \ee
    \sum_{T \in {\mathcal T}}{\mathbb P}(T)
    =\ee.
\end{equation}
The relation with the critical Poisson branching process
can easily be exploited further
(see, e.g., \cite[Theorem~2.1]{BCHS99}) to
obtain
\begin{equation}
    \sum_{n=0}^\infty f_n  z^{n}
    =
    (K  z)^{-1}\sum_{n=1}^\infty
    \frac{n^{n-1}}{n!}
    (\Degree  z)^{n} .
\end{equation}
The series on the right-hand side converges if and only if
$|\Degree\ee z | \le 1$, by Stirling's
formula, and hence
\begin{equation}
\label{fnlim}
    \lim_{n \to \infty} f_n^{1/n} = \Degree \ee = \frac{1}{z_0}.
\end{equation}

Let ${\mathcal L}_n$ denote the set of $n$-bond lattice trees containing
the origin; its cardinality is $t_n$.  We
will use the fact, proved in \cite[(5.5)]{BCHS99},
that for every $L \in {\mathcal L}_n$,
\begin{equation}
\label{Mnsum}
    \sum_{(T,\varphi) \in {\mathcal M}_n : \varphi (T)=L}
    \prod_{i\in T}\frac{1}{\xi_i!}
    =1.
\end{equation}
The proof of \eqref{Mnsum}
in \cite{BCHS99} is given for the nearest-neighbour model,
but it applies without change also to the spread-out model.
By summing \eqref{Mnsum} over $L \in {\mathcal L}_n$,
we obtain
\begin{equation}
\label{ft}
    t_n  \le f_n,
\end{equation}
and hence $\tau \le \lim_{n \to \infty}f_n^{1/n}
=\Degree \ee$.  This gives the inequality $z_c \ge z_0$, which
is weaker than the inequality $z_c^{(a)} \ge z_0$ that we seek
in Lemma~\ref{lem:z}.

\begin{proof}[Proof of Lemma~\ref{lem:z}]
The inequality $z_c \ge z_c^{(a)}$ follows from $t_n \le a_n$,
and the equality
$z_0=(\Degree\ee)^{-1}$ holds by definition, so it suffices to prove
that $z_c^{(a)} \ge z_0$.  By \eqref{fnlim}, for this it suffices
to prove that
\begin{equation}
\label{anfn}
    a_n \le f_n.
\end{equation}
To prove this, we adapt the proof of \eqref{Mnsum} from \cite{BCHS99}.

The first step involves a unique determination of a tree
structure within a lattice animal.  For this, we
order all bonds in the infinite lattice
lexicographically.  Also, we regard a bond
as an arc joining the vertices of its endpoints, and we order
the two halves of this arc as minimal and maximal.  These orderings
are fixed once and for all.
Given a lattice animal $A$, suppose that it contains $c$ cycles.
Choose the minimal bond whose removal would break a cycle,
and remove its minimal half from the animal.  Repeat this until no cycles
remain.  The result is a kind of lattice tree, which we will call
the \emph{cut-tree} $A^*$,
in which $c$ leaves are endpoints of half edges.
See Figure~\ref{fig:cut-tree}.
Let ${\mathcal A}_n$ denote the set of $n$-bond lattice animals
that contain the origin.
Let ${\mathcal A}^*_n$ denote the set of $n$-bond cut-trees
that can be produced from a lattice animal in ${\mathcal A}_n$
by this procedure.
By construction, lattice animals and cut-trees are in one-to-one
correspondence, so ${\mathcal A}_n^*$ has cardinality $a_n$.

\begin{figure}
\begin{center}
\setlength{\unitlength}{0.0125in}%
\begin{picture}(300,127)(60,680)
\thicklines
\put( 80,760){\line( 1, 0){ 20}}
\put(100,760){\line( 0,-1){ 20}}
\put(100,740){\line( 0,-1){ 20}}
\put(100,720){\line( 1, 0){ 20}}
\put(120,720){\line( 1, 0){ 20}}
\put(140,720){\line( 1, 0){ 20}}
\put(160,720){\line( 0, 1){ 20}}
\put(160,740){\line( 1, 0){ 20}}
\put( 80,780){\line( 0,-1){ 60}}
\put( 80,720){\line( 1, 0){ 20}}
\put(100,760){\line( 1, 0){ 20}}
\put(100,740){\line( 1, 0){ 40}}
\put(120,740){\line( 0,-1){ 20}}
\put(120,760){\line( 1, 0){ 20}}
\put(160,720){\line( 0,-1){ 20}}
\put( 60,760){\line( 1, 0){ 20}}
\put( 60,760){\line( 0,-1){ 20}}
\put( 60,740){\line(-1, 0){ 20}}
\put(120,760){\line( 0, 1){ 40}}
\put(100,720){\line( 0,-1){ 20}}
\put(140,740){\line( 0, 1){ 20}}
\put(140,720){\line( 0,-1){ 40}}
\put(140,700){\line( 1, 0){ 20}}
\put(160,760){\line( 0,-1){ 20}}
\put(160,760){\line( 1, 0){ 40}}
\put( 50,720){\makebox(0,0)[lb]{\raisebox{0pt}[0pt][0pt]{$A$}}}
\put(260,760){\line( 1, 0){ 20}}
\put(280,760){\line( 0,-1){ 10}}
\put(280,740){\line( 0,-1){ 10}}
\put(280,720){\line( 1, 0){ 20}}
\put(300,720){\line( 1, 0){ 20}}
\put(320,720){\line( 1, 0){ 20}}
\put(340,720){\line( 0, 1){ 20}}
\put(340,740){\line( 1, 0){ 20}}
\put(260,780){\line( 0,-1){ 50}}
\put(260,720){\line( 1, 0){ 20}}
\put(280,760){\line( 1, 0){ 20}}
\put(280,740){\line( 1, 0){ 40}}
\put(300,740){\line( 0,-1){ 20}}
\put(300,760){\line( 1, 0){ 20}}
\put(340,720){\line( 0,-1){ 20}}
\put(240,760){\line( 1, 0){ 20}}
\put(240,760){\line( 0,-1){ 20}}
\put(240,740){\line(-1, 0){ 20}}
\put(300,760){\line( 0, 1){ 40}}
\put(280,720){\line( 0,-1){ 20}}
\put(320,740){\line( 0, 1){ 20}}
\put(320,720){\line( 0,-1){ 10}}
\put(320,700){\line( 0,-1){ 20}}
\put(320,700){\line( 1, 0){ 20}}
\put(340,760){\line( 0,-1){ 20}}
\put(340,760){\line( 1, 0){ 40}}
\put(230,720){\makebox(0,0)[lb]{\raisebox{0pt}[0pt][0pt]{$A^*$}}}
\end{picture}
\end{center}
\caption{\label{fig:cut-tree}
A lattice animal $A$ and its associated cut-tree $A^*$.}
\end{figure}
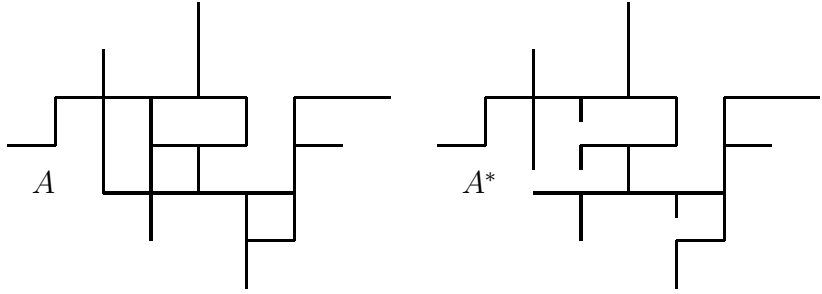

We may regard the edges of $T \in {\mathcal T}$ as directed away from
the root, and we write a directed edge as $(i,i')$.
Given $A^*\in {\mathcal A}_n^*$ and $(T,\varphi)
\in {\mathcal M}_n$, we say that $\varphi(T)=A^*$ if (i) each bond
in $A$ is the image of a unique edge in $T$ under $\varphi$,
and if, in addition, (ii) if $(b^+,b^-)$ is a directed bond in $A$ from which
the half-bond containing $b^-$ is removed in $A^*$, and if the edge
of $T$ that is mapped by $\varphi$ to $(b^+,b^-)$ is $(i,i')$, then
$i'$ is a leaf of $T$.
Roughly speaking, the
condition $\varphi(T)=A^*$ means that the mapping $\varphi$ ``folds''
$T$ over $A^*$ in such a way that the tree structure of $T$ is
preserved in $A^*$.
We claim that for every
$A^* \in {\mathcal A}^*_n$,
\begin{equation}
\label{Astarsum}
    \sum_{(T,\varphi) \in {\mathcal M}_n : \varphi (T)=A^*}
    \prod_{i\in T}\frac{1}{\xi_i!}
    =1.
\end{equation}
This implies that
\begin{align}
    a_n & =
    \sum_{(T,\varphi) \in {\mathcal M}_n : \varphi (T)\in {\mathcal A}^*_n}
    \prod_{i\in T}\frac{1}{\xi_i!}
    \le
    \sum_{(T,\varphi) \in {\mathcal M}_n }
    \prod_{i\in T}\frac{1}{\xi_i!}
    = f_n,
\end{align}
which is the
required inequality \eqref{anfn}.  Thus it suffices to prove \eqref{Astarsum}.

To prove \eqref{Astarsum}, we adapt the proof of \eqref{Mnsum}
from \cite{BCHS99}, as follows.
Let $b_0$ be the degree of $0$ in $A^*$,
and given a nonzero vertex $x \in A^*$, let $b_x$
be the degree of $x$ in $A^*$ minus $1$ (the forward degree of $x$).
Then the set $\{ b_x : x \in A^*\}$ (with repetitions) must be equal to the
set $\{\xi_i : i \in T\}$ (with repetitions) for any $T$ such that $\varphi(T)=A^*$.
Defining $\nu(A^*)$ to be the cardinality
of  $\{(T,\varphi) : \varphi(T)=A^* \}$, \eqref{Astarsum}
is therefore equivalent to
\begin{equation}
\label{wsaltsuff}
    \nu(A^*)  = \prod_{x \in A^*} b_x!.
\end{equation}

We prove \eqref{wsaltsuff} by induction on the number $N$ of generations
of $A^*$, i.e., the number of bonds or half-bonds in
the longest self-avoiding path
in $A^*$ starting from the origin.  The identity \eqref{wsaltsuff} clearly
holds if $N=0$.  Our induction hypothesis
is that \eqref{wsaltsuff} holds if there are $N-1$ or fewer generations.
Suppose $A^*$ has $N$ generations, and let $A^*_1, \ldots, A^*_{b_0}$ denote
the cut-trees resulting by deleting from $A^*$ the origin and
all bonds and half-bonds incident on the origin.
We regard each $A^*_a$ as rooted at the non-zero vertex
in the corresponding deleted bond.  It suffices to show
that $\nu(A^*) = b_0! \prod_{a=1}^{b_0} \nu(A^*_a)$, since each $A^*_a$ has fewer
than $N$ generations.

To prove this, we note that each pair $(T,\varphi)$
with $\varphi(T)=A^*$ induces a set of $(T_a,\varphi_a)$ such that
$\varphi_a(T_a)=A^*_a$.  This correspondence is
$b_0!$ to $1$, since $(T,\varphi)$
is determined by the set of $(T_a,\varphi_a)$, up to
permutation of the branches of $T$ at its root.
This proves
$\nu(A^*) = b_0! \prod_{a=1}^{b_0} \nu(A^*_a)$, and completes the proof
of the lemma.
\end{proof}

\begin{lemma}
\label{lem:n}
For the nearest-neighbour or spread-out models (the latter in all
dimensions $d \ge 1$), for each fixed $n \ge 0$,
\begin{equation}
\lim_{\Degree\to\infty} \frac{t_n}{\Degree^{n}}
= \sum_{T\in {\mathcal T}_n}\prod_{i\in T}\frac{1}{\xi_i!}.
\end{equation}
\end{lemma}

\begin{proof}
By \eqref{Mnsum},
\begin{align}
    t_n & =
    \sum_{(T,\varphi) \in {\mathcal M}_n : \varphi (T)\in {\mathcal L}_n}
    \prod_{i\in T}\frac{1}{\xi_i!}
    = \sum_{T \in {\mathcal T}_n}
    \prod_{i\in T}\frac{1}{\xi_i!}
    \sum_{\varphi \in \Phi(T) : \varphi (T)\in {\mathcal L}_n} 1.
\label{tnsum}
\end{align}
Given $T \in {\mathcal T}_n$, the cardinality of
$\Phi(T)$ is $\Degree^n$, so there are at most $\Degree^n$ nonzero terms
in the above sum over $\varphi$.  On the other hand, there are at least
$\Degree (\Degree - 1)\cdots(\Degree -n+1)$ nonzero terms.
To see this, consider the mapping $\varphi$ of $T$ to proceed in
a connected fashion to map the edges of $T$ one by one to bonds in
$\Z^d$, starting from the root.  The first edge of $T$ can be mapped to
any one of $K$ possible bonds.  The second edge of $T$ includes one
of the vertices of the first edge, and to avoid the image of the other
vertex of the first edge, it can be mapped to any one of $K-1$ possible
edges.  In this way, as
$\varphi$ proceeds from the root
to map vertices of $T$ into $\Zd$, the restriction that the image contain
$n+1$ distinct vertices
allows $K$ choices for the first bond, $K-1$ choices for
the second bond, at least $K-2$ for the third, at least $K-3$ for
the fourth, and so on.
This implies that
\begin{equation}
    \Degree (\Degree - 1)\cdots(\Degree -n+1)
    \sum_{T \in {\mathcal T}_n}
    \prod_{i\in T}\frac{1}{\xi_i!}
    \le
    t_n
    \le
    \Degree^n
    \sum_{T \in {\mathcal T}_n}
    \prod_{i\in T}\frac{1}{\xi_i!},
\end{equation}
and the desired conclusion follows.
\end{proof}

\begin{proof}[Proof of Proposition~\ref{prop:z0}.]
By \eqref{ft}, \eqref{fdef}
and \eqref{z0def},
\begin{equation}
    t_nz_0^n \le f_nz_0^n
    =
    \ee^{-n} \sum_{T\in {\mathcal T}_n}\prod_{i\in T}\frac{1}{\xi_i!},
\end{equation}
which is independent of $K$.
Also, by \eqref{f1},
$\sum_{n=0}^\infty f_nz_0^n=\ee$.
Hence, by
Lemma~\ref{lem:n} and the dominated convergence theorem, we have
\begin{align}
\lim_{\Degree \to\infty}g(z_0)
&= \sum_{n=0}^{\infty}
\lim_{\Degree\to\infty}t_n z_0^n%
        =\sum_{n=0}^{\infty}
        \left( \sum_{T\in {\mathcal T}_n}\prod_{i\in T}\frac{1}{\xi_i!}
        \right) \ee^{-n}
        =
        \sum_{n=0}^{\infty}f_n z_0^n
        = \ee .
\end{align}
\end{proof}

\section*{Acknowledgements}

The work of YMM was supported in part by CONACYT of Mexico.
The work of GS was supported in part by NSERC of
Canada.

\end{document}